\documentclass[11pt]{article}%

\usepackage{tikz}

\usepackage[latin1]{inputenc}
\usepackage{natbib}

\usepackage[english]{babel}
\usepackage{pst-all}
\usepackage{amssymb}
\usepackage{amsmath,enumerate,rotate}
\usepackage{amssymb,latexsym}
\usepackage{epsfig}
\usepackage{graphicx, graphics,float,subfigure}
\usepackage{amsthm}

\usepackage{multirow}
\usepackage{tabularx,ragged2e}
\usepackage{booktabs}

\usepackage[colorlinks=true,linkcolor=blue]{hyperref}%
\usepackage{graphicx}
\hypersetup{urlcolor=blue, citecolor=blue, colorlinks=true, linkcolor=blue} 
\usepackage[margin=1in]{geometry}
\usepackage{verbatim}

\floatstyle{plaintop}
\restylefloat{table}
\usepackage[tableposition=top]{caption}
 \usepackage{mathtools}

\usepackage{soul}
\usepackage{color}
\usepackage[mathcal]{eucal}
\usepackage{mathrsfs} 
\usepackage{verbatim}
\usepackage{url}
\usepackage{setspace}
\usepackage{mathtools} 

\theoremstyle{plain}
\newtheorem{theorem}{Theorem}
\newtheorem*{theorem*}{Theorem}

\newtheorem*{corollary*}{Corollary}
\newtheorem*{definition*}{Definition}

\theoremstyle{remark}

\def\S{\mathbb{S}}
\def\R{\mathbb{R}}

\allowdisplaybreaks

\begin{document}

	

\fontsize{12}{14pt plus.8pt minus .6pt}\selectfont \vspace{0.8pc}
\centerline{\LARGE \bf Rudin Extension Theorems on Product Spaces, 
}
\vspace{16 pt} 

\centerline{\LARGE \bf 
Turning Bands,
}
\vspace{16pt} \centerline{\LARGE \bf and Random Fields on Balls cross Time }

\vspace{.8cm} 
\begin{center}
	{\large {\sc Emilio Porcu,}}\footnote{\baselineskip=10pt
	    Department of Mathematics, \\
	    Research And Data Intelligence Support Center R-DISC,\\
	    Khalifa University, Abu Dhabi, United Arab Emirates, \\
		$\&$ School of Computer Science and Statistics, \\
		Trinity College Dublin. \\
		$\&$ Millennium Nucleus Center for the Discovery of Structures in Complex Data, Chile.
 \\ 
		E-mail: emilio.porcu@ku.ac.ae  \\
	} 
	{\large {\sc Samuel F. Feng,}}\footnote{Department of Mathematics, \\
	    Research And Data Intelligence Support Center R-DISC,\\
	    Khalifa University, Abu Dhabi, United Arab Emirates, \\
	    samuel.feng@ku.ac.ae}
	{\large {\sc Xavier Emery,}}\footnote{Deparment of Mining Engineering, \\ University of Chile \\ $\&$ Advanced Mining Technology Center, \\ University of Chile \\ xemery@ing.uchile.cl  }
{\large and {\sc Ana Paula Peron,}}\footnote{Department of Mathematics, \\
Institute of Mathematical and Computer Sciences,\\
	    University of S\~ao Paulo, Brazil, \\ apperon@icmc.usp.br  }

\def\dd{d^{\prime}}

\begin{abstract}
Characteristic functions that are radially symmetric have a dual interpretation, as they can be used as the isotropic correlation functions of spatial random fields. Extensions of isotropic correlation functions from balls into $d$-dimensional Euclidean spaces, $\R^{d}$, have been understood after Rudin. Yet, extension theorems on product spaces are elusive, and a counterexample provided by Rudin on rectangles suggest that the problem is quite challenging. \\
This paper provides extension theorem for {multiradial} characteristic functions that are defined in balls embedded in $\R^d$ cross, either $\R^{\dd}$ or the unit sphere $\S^{\dd}$ embedded in $\R^{\dd+1}$, for any two positive integers $d$ and $\dd$. We then examine Turning Bands operators that provide bijections between the class of multiradial correlation functions in given product spaces, and multiradial correlations in product spaces having different dimensions. \\
The combination of extension theorems with Turning Bands provides a connection with random fields that are defined in balls cross linear or circular time.

\vspace{2cm}

\end{abstract}
\end{center}

\newpage

\def\B{\mathbb{B}}
\def\dd{d^{\prime}}

\section{Introduction}

\subsection{Extension Problems}

The study of positive definite functions traces back to \cite{hilbert} and \cite{caratheodory}. The extension of such functions from an original space to a wider space has preoccupied mathematicians since the early $40$ies, and we refer to \cite{krein1940}, \cite{calderon} as well to the {\em tour de force} by  \cite{rudin1963extension, rudin1970}. \\
While \cite{rudin1970}'s extension theorem refers to positive definite functions that are defined over some compact interval on the real line, subsequent researches have been devoted to attain extension theorems for multidimensional spaces, the radial case being of special interest to both probability theory and spatial statistics: a {\em normalized} positive definite radial function defined in the $d$-dimensional Euclidean space, $\R^d$, is the characteristic function of a random vector in $\R^d$, as well as the correlation function of a Gaussian random field that is stationary and isotropic ({\em i.e.}, radially symmetric) in $\R^d$ \citep{schoenberg2, matheron}. For a comprehensive review on the extension problem, the reader is referred to \cite{sasvari2005}. \\
\cite{gneiting-sasvari} proved that any positive definite radial function defined in a ball embedded in 
$\R^d$, $d>1$, which is not necessarily continuous at the origin, admits an extension to a positive definite radial function in $\R^d$. Generalizations to product spaces under radiality have not been considered so far. The difficulty of the problem is actually confirmed by the counterexample produced by \cite{rudin1963extension}, who proves that positive definite functions defined on rectangles in $\R^2$ might not have a positive definite extension to the whole plane. Theorem 4.3.6 in \cite{sasvari1994positive} shows that extensions in product spaces are possible, but the radiality of the extension is not necessarily preserved. Example 4.2.9(b) in \cite{sasvari1994positive} for a strip embedded in $\R^2$ suggests that such extensions might be possible under suitable regularity assumptions on the function to be extended. 

\subsection{Turning Bands Operator}
\label{TB}

\cite{matheron1965variables, matheron1971} proposed the illustrative terms {\em mont{\'e}e} (upgrading) and {\em descente} (downgrading) to describe operators which, when applied to suitable radially symmetric characteristic functions in $\R^d$, yield radially symmetric characteristic functions in $\R^{\dd}$, for $\dd$ being larger or smaller than $d$. \cite{Wendland} adopted the name {\em walk through dimensions} to describe the role of these operators and showed their effect on radial characteristic functions in terms of smoothness. \\

Related to the mont\'ee, \cite{matheron1972, matheron} 
defined the so-called {\em Turning Bands} operator allowing for a bijection between the class of symmetric characteristic functions on the real line, and the class of radially symmetric characteristic functions in $\R^d$, for $d>1$. The impact of Turning Bands can be appreciated in subsequent developments in probability theory \citep{eaton, daley-porcu, cambanis}, spatial statistics \citep{gneiting-sasvari, gneiting1999isotropic, Porcu:Gregori:Mateu:2006}, and on geostatistical simulations \citep{lantuejoul, emery0, emery1}, to mention a few.

\subsection{Our Contribution}
We consider multiradial characteristic functions ({\em i.e.}, componentwise isotropic correlation functions) that are defined in product spaces. Specifically, we consider the case of the product of a $d$-dimensional ball with radius $1/2$, $\B_d$, with the $\dd$-dimensional Euclidean space, $\R^{\dd}$, and a positive definite function that is radial in the ball and radial in $\R^{\dd}$. {This case is of interest in application fields such as meteorology, climatology, geodesy and geophysics, where there is a need to model atmospheric, gravimetric, seismic or magnetic data indexed by latitude, longitude, altitude/depth and time, e.g., satellite data, data located in the Earth's mantle or on the Earth surface \citep{Hofmann, gillet, Meschede, Ern, Finlay, Xu}.}
We also consider the product space $\R^d \times \S^{\dd}$, for $\S^{\dd}$ the $\dd$-dimensional unit sphere embedded in $\R^{\dd+1}$; here, by multiradiality we mean that the positive definite function is radial in $\R^{d}$, and depends on the inner product on $\S^{\dd}$. 
For these two cases, we prove that Rudin's extensions are indeed possible. Our proof is based on results of independent interest: we prove characterizations of these types of positive definite functions in terms of partial Fourier transforms.  \\

Additionally, we extend Matheron's Turning Bands operators to the aforementioned product spaces. While the first case $\B_d \times \R^{\dd}$ can be proved through direct inspection, the case $\R^{d} \times \S^{\dd}$ comes to us as a surprise. 
Merging extensions with Turning Bands over product spaces, we find a connection related to Gaussian random fields that are defined over balls cross linear or circular time. Specifically, we provide new classes of nonseparable correlation functions that are isotropic over balls and symmetric on the real line, or circularly symmetric on the circle. \\ 

The outline of the paper is the following. Section \ref{sec3} provides  background material. Section \ref{sec4} challenges the extension problem for characteristic functions. Results related to new Turning Bands-type operators over product spaces are reported in Section \ref{sec5}. Section \ref{sec6} connects Sections \ref{sec4} and \ref{sec5} to illustrate a method for constructing correlation functions that are isotropic over the ball $\B_d$, and symmetric over linear time, $\R$, or circularly symmetric over circular time, $\S^1$.

\section{Background and Notation}  \label{sec3}
\subsection{Schoenberg Classes} \label{sec3a}
We consider the class ${\cal P}(\R^d, \|\cdot\|)$ of {real-valued} continuous mappings $f: [0,\infty) \to \R$ with $f(0)=1$ such that $f (\|\cdot\|)$ is positive definite and radial in $\R^d$, with $\|\cdot\|$ denoting the Euclidean norm. Such functions are termed isotropic in spatial statistics, and the reader is referred to \cite{daley-porcu}, with the references therein, for a comprehensive treatment. We also define ${\cal P}(\R^{\infty}, \|\cdot\|):= \bigcap_{d \ge 1} {\cal P}(\R^d, \|\cdot\|)$. A characterization of the class ${\cal P}(\R^d, \|\cdot\|)$ is available thanks to \cite{schoenberg2}: a continuous mapping $f$ defined on the positive real line belongs to ${\cal P}(\R^d, \|\cdot\|)$ if and only if 
\begin{equation}
\label{schoenberg1} f(x) = \int_{[0,\infty)} \Omega_d (x \xi) {\rm d} F (\xi), \qquad x \ge 0,
\end{equation}
where $F$ is a probability measure {on the positive real line} and $\Omega_d(x) = \mathbb{E} \left ( {\rm e}^{\mathsf{i} x \langle \mathbf{e}_1, \boldsymbol{\eta} \rangle}\right ) $, with $\mathsf{i}$ the imaginary unit, $\mathbf{e}_1$ a unit vector in $\R^d$, and $\boldsymbol{\eta}$ a random vector that is uniformly distributed on the spherical shell $\mathbb{S}^{d-1} \subset \R^d$. The kernel $\Omega_d$ has several representations. Here, we invoke direct inspection to write $\Omega_d$ as 
\begin{equation}
\label{Omega}  \Omega_d(x) = \frac{\Gamma(d/2) J_{d/2-1} (x) }{\left ( \frac{x}{2}\right )^{d/2-1}}= \sum_{n=0}^{\infty} \frac{\Gamma(d/2) \left (-  \frac{x^2}{4}\right )^{n} }{\Gamma(d/2+n) n!}, \qquad t \ge 0,
\end{equation}
where $J_{\nu}$ stands for the Bessel function of the first kind of order $\nu$.
Since $|\Omega_d(x)| \le \Omega_d(0) =1$ \citep{daley-porcu}, one has $f(0)=1$. Clearly, $f \in {\cal P}(\R^d, \|\cdot\|)$  is the {radial part of the } characteristic function of a random vector, $\boldsymbol{X}$, that is equal (in distribution) to the product between the random vector $\boldsymbol{\eta}$  and a random variable, $Z$, independent of $\boldsymbol{\eta}$ and distributed according to $F$. 
The class ${\cal P}(\R^d, \|\cdot\|)$ is nested, with the inclusions relation ${\cal P}(\R, |\cdot|) \supset{\cal P}(\R^2, \|\cdot\|) \supset \ldots \supset {\cal P}(\R^{\infty}, \|\cdot\|)$ being strict.

We now define the class ${\cal P}(\R^d \times \R^{\dd}, \|\cdot\|,\|\cdot \|)$ of {real-valued} continuous mappings $\varphi:[0,\infty)^2 \to \R$ with $\varphi(0,0)=1$, such that $\varphi \left ( \|\cdot\|, \|\cdot\|\right ) $ is positive definite in $\R^{d } \times \R^{\dd}$. Such functions $\varphi$ are called multiradial in \cite{porcu-mateu-christakos}, and arguments therein show that $\varphi \in {\cal P}(\R^d \times \R^{\dd}, \|\cdot\|,\|\cdot \|)$ for given $d,\dd \in \mathbb{N}$ if and only if 
\begin{equation}
\label{schoenberg2} \varphi(x,t) = \int_{[0,\infty)^2} \Omega_d (x \xi) \Omega_{\dd} (t v) {\rm d} H (v,\xi), \qquad x,t \ge 0,
\end{equation}
with $H$ being a probability measure on the positive quadrant of $\R^2$. Clearly, $\varphi \in {\cal P}(\R^d \times \R^{\dd}, \|\cdot\|,\|\cdot \|)$ implies $\varphi(\cdot,0)$ and $\varphi(0,\cdot)$ to belong to ${\cal P}(\R^d, \|\cdot\|)$ and ${\cal P}( \R^{\dd}, \|\cdot\|)$, respectively. The class ${\cal P}(\R^d \times \R^{\dd}, \|\cdot\|,\|\cdot \|)$ is nested in the same way as ${\cal P}(\R^d , \|\cdot\|)$. 

The seminal paper by \cite{schoenberg1942} characterizes the class ${\cal P}(\S^d, \theta)$ of real-valued continuous mappings $\psi:[0,\pi] \to \R$ with $\psi(0)=1$, such that $\psi(\theta(\cdot,\cdot))$ is positive definite on the unit sphere, $\S^{d} \subset \R^{d+1}$, and where $\theta$ denotes the geodesic distance, defined as $\theta(\boldsymbol{x},\boldsymbol{y})= \arccos (\langle \boldsymbol{x},\boldsymbol{y} \rangle ) $, $\boldsymbol{x},\boldsymbol{y}\in \S^d$, with $\langle \cdot, \cdot\rangle$ being the inner product in $\R^{d+1}$. Specifically,  \cite{schoenberg1942} shows that $\psi \in {\cal P}(\S^d, \theta)$ for a given $d \in \mathbb{N}$, $d>1$, if and only if 
\begin{equation}
\label{schoenberg3a}  \psi(\theta) = \sum_{n=0}^{\infty} b_{n,d} \frac{{\cal G}_{n}^{(d-1)/2} (\cos \theta)}{{\cal G}_{n}^{(d-1)/2} (1)}, \qquad \theta \in [0,\pi], 
\end{equation}
where ${\cal G}_n^{\lambda}$ is the $n$th Gegenbauer polynomial of order $\lambda > 0$ \citep{szego}, and where $\{ b_{n,d}\}_{n=0}^{\infty}$ is a uniquely determined sequence of nonnegative coefficients summing up to one, {\em i.e.}, a probability mass system. The decomposition {\eqref{schoenberg3a} remains valid for $d=1$ provided that the normalized Gegenbauer polynomials are replaced by the Chebyshev polynomials of the first kind.}

Finally, the class ${\cal P}( \R^d \times \S^{\dd}, \|\cdot\|,\theta)$ of continuous {real-valued} mappings $\psi:[0,\infty) \times [0,\pi]  \to \R$ {such that $\psi(0,0)=1$ and $\psi(\|\cdot\|,\theta)$ is positive definite} has been characterized by \cite{berg-porcu} through uniquely determined expansions of the type 
\begin{equation}
\label{schoenberg3} \psi(x,\theta)=   \sum_{n=0}^{\infty} b_{n,\dd}(x) \frac{{\cal G}_{n}^{(d-1)/2} (\cos \theta)}{{\cal G}_{n}^{(d-1)/2} (1)}, \qquad x\ge 0, \; \theta \in [0,\pi], 
\end{equation}
where $\{ b_{n,\dd} (\cdot)\}_{n=0}^{\infty}$ is a uniquely determined sequence of members of ${\cal P}(\R^d,\|\cdot\|)$ with the additional requirement that $\sum_{n=0}^{\infty} b_{n,\dd}(0) =1$ {(for $\dd=1$, one has to replace the normalized Gegenbauer polynomials in \eqref{schoenberg3} by the Chebyshev polynomials of the first kind)}. We follow \cite{daley-porcu} and \cite{berg-porcu} to call the sequences $\{ b_{n,d} \}_{n=0}^{\infty}$
in (\ref{schoenberg3a}) and $\{ b_{n,\dd} (\cdot)\}_{n=0}^{\infty}$ in (\ref{schoenberg3}), $d$-Schoenberg sequences of coefficients, and $\dd$-Schoenberg sequences of functions, respectively.

\subsection{Rudin's Extension}
\def\B{\mathbb{B}}

\cite{rudin1970} considers the class ${\cal P} (\B_d,\|\cdot\|) $ of continuous functions $g:[0,1) \to \R$ with $g(0)=1$ such that the composition $g (\|\cdot\|)$ is positive definite
in the open ball, $\B_d = \{ \boldsymbol{x} \in \R^d, \; \|\boldsymbol{x}\| < 1/2 \}$, embedded in $\R^d$. Clearly, $f \in {\cal P} (\R^d,\|\cdot\|)$ implies the corresponding restriction to $[0,1)$ to belong to ${\cal P} (\B_d,\|\cdot\|)$. The opposite is not obvious and has been shown by \cite{rudin1970}. We state the result formally for the convenience of the reader. 
\begin{theorem}[\citealp{rudin1970}]
\label{rudin} Let $d$ be a positive integer. Let $g $ belong to the class ${\cal P} (\B_d,\|\cdot\|) $. Then, there exists a continuous mapping $f: [0,\infty) \to \R$ such that $f$ belongs to ${\cal P} (\R^d,\|\cdot\|) $. Additionally, $f(t)=g(t)$ for all $t \in [0,1)$.
\end{theorem}
Rudin's beautiful result provides an important message to the spatial statistics community. The class of isotropic correlation functions in $\B_d$ is not larger than the Schoenberg class ${\cal P}(\R^d, \|\cdot\|)$ restricted to $\B_d$. Hence, the function belonging to the class ${\cal P}(\B_d, \|\cdot\|)$ enjoys the scale mixture representation in Equation (\ref{schoenberg1}).

Surprisingly, analogues of Rudin's extensions for the case of product spaces are rare. \cite{sasvari1994positive} provides a generalization by considering the product space $V_2 \times G_1$, with $G_1$ being a commutative group, and $V_2$ a subgroup of an arbitrary group, $G_2$. Unfortunately, the extension to $G_2 \times G_1$ is not necessarily radial, and no extensions are up to now available for the classes introduced in Section \ref{sec3a}. Throughout, we use the notation ${\cal P}(\B_d \times \R^{\dd}, \|\cdot\|,\|\cdot\|)$ and ${\cal P}(\B_d \times \S^{\dd}, \|\cdot\|,\theta)$ in analogy with the classes that have been previously defined. 

\subsection{The Turning Bands Operator} 

We now revisit the Turning Bands operator introduced in Section \ref{TB}, which gives a bijection between ${\cal P}(\R,|\cdot|)$ and ${\cal P}(\R^d,\|\cdot\| )$. 
A slight change of notation is needed for a neater exposition. We use the subindex $d$ to indicate a function $\varphi_d$ belonging to the class ${\cal P}(\R^d,\|\cdot\| )$, $d>1$. \cite{matheron} proved that 
\begin{equation}
\label{TB1} \varphi_d(t) = \frac{2 \Gamma(d/2) }{\sqrt{\pi} \Gamma ((d-1)/2)} \frac{1}{t} \int_{0}^{t} \varphi_1(u) \left ( 1- \frac{u^2}{t^2}\right )^{(d-3)/2}   {\rm d} u,  
\end{equation} 
for $\varphi_1 \in {\cal P}(\R,|\cdot|)$. 
\cite{gneiting1999isotropic} uses Equation (\ref{TB1}) in concert with Theorem \ref{rudin} to prove that the Turning Bands operator also provides a bijection between the classes ${\cal P}(\B_1,|\cdot|)$ and ${\cal P}(\B^d,\|\cdot\| )$.

\section{The Extension Problem in Product Spaces} \label{sec4}

We start by illustrating a result that provides the basis to a constructive proof of our extension theorems. 

\begin{theorem}
\label{lemma}. (a). Let $d,\dd$ be two positive integers. Let $\varphi: [0,1) \times [0,\infty) \to \R$ be continuous with $\varphi(0,0)=1$, and such that $\varphi(x,\| \cdot\|)$ is absolutely integrable in $\R^{\dd}$ for every $x\ge 0$. Then, $\varphi$ belongs to the class ${\cal P}(\B_d \times \R^{\dd}, \|\cdot\|,\|\cdot\|)$ if and only if the mapping $\varphi_{\boldsymbol{\omega}}$, defined through 
\def\bw{\boldsymbol{\omega}}
\begin{equation}
\label{lemma_a} \varphi_{\boldsymbol{\omega}} (x) :=  \int_{\R^{\dd}} {\rm e}^{-\mathsf{i} \langle \boldsymbol{\omega}, \boldsymbol{y} \rangle} \varphi (x, \|\boldsymbol{y}\|) {\rm d} \boldsymbol{y},  \quad x\in[0,1),
\end{equation} 
is such that $\varphi_{\boldsymbol{\omega}} (x)/\varphi_{\boldsymbol{\omega}} (0)$ belongs to the class ${\cal P}(\B_d,\| \cdot\|)$ for every $\boldsymbol{\omega} \in \R^{\dd}$. \\ 
(b). Let $\psi: [0,1) \times [0,\pi] \to \R$ be continuous with $\psi(0,0)=1$. Then, $\psi \in {\cal P} (\B_d \times \S^{\dd}, \|\cdot\|,\theta)$ if and only if the functions $ b_{n,\dd}(\cdot) $, defined through 
\begin{equation}
\label{lemma_b} b_{n,\dd}(x) := c_{\dd} \int_{0}^{\pi} \psi(x,\theta) (\sin \theta)^{\dd-1} {\cal G}_{n}^{(\dd-1)/2} (\cos \theta) {\rm d} \theta, \qquad {x \in [0,1),} 
\end{equation}
form a sequence of members of ${\cal P}(\B_d,\| \cdot\|)$ for all $n \in \mathbb{N}_0$ with the additional requirement that $\sum_{n=0}^{\infty} b_{n,\dd}(0) =1$. Here, $c_{\dd}$ is a strictly positive normalization constant that depends on $\dd$.  When $\dd=1$, then Gegenbauer polynomials in \eqref{lemma_b} need be replaced by Chebyshev polynomials.
\end{theorem}
\begin{proof}
(a). The necessity is proved by direct construction. If $\varphi \in {\cal P}(\B_d \times \R^{\dd}, \|\cdot\|,\|\cdot\|)$, then a direct application of the Schur {product} theorem shows that the {complex-valued} mapping
$$ 
( \boldsymbol{x},\boldsymbol{y}) \mapsto {\rm e }^{-\mathsf{i} \langle \boldsymbol{\omega}, \boldsymbol{y} \rangle} \varphi (\|\boldsymbol{x}\|, \|\boldsymbol{y}\|), \qquad (\boldsymbol{x}, \boldsymbol{y}) \in \B_d \times \R^{\dd}, 
$$
is positive definite in $\B^{d } \times \R^{\dd}$ for every $\boldsymbol{\omega} \in \R^{\dd}$. 
Since positive definite functions are a convex cone that is closed under scale mixtures, we get that $\varphi_{\boldsymbol{\omega}}$ as defined through (\ref{lemma_a}) is positive definite in $\B_d$. The necessity part of the proof is completed by noting that $\varphi_{\boldsymbol{\omega}}$ is well-defined thanks to the integrability condition on $\varphi$ in concert with the fact that the complex exponential is uniformly bounded by one. \\

To prove the sufficiency, we let $\varphi_{\boldsymbol{\omega}}$ as defined at (\ref{lemma_a}), and let $\varphi_{\boldsymbol{\omega}}/\varphi_{\boldsymbol{\omega}}(0)$ be a member of the class ${\cal P}(\B_d,\|\cdot\|)$ for every $\boldsymbol{\omega} \in \R^{\dd}$.  Equation (\ref{lemma_a}) in concert with classical Fourier inversion allow to write
\begin{equation}
\label{lemma_c} {\varphi}(x,y)= {\frac{1}{(2\pi)^{\dd}}} \int_{\R^{\dd}} {\rm e}^{\mathsf{i} \langle \boldsymbol{\omega}, y \boldsymbol{e}_1^{\prime} \rangle} \varphi_{\boldsymbol{\omega}}(x) {\rm d} \boldsymbol{\omega}, \qquad x \in [0,1), y \in [0,\infty). 
\end{equation}
with 
$\boldsymbol{e}_1^{\prime}$ a unit vector in $\R^{\dd}$. Note that $\varphi$ does not depend on the particular choice of vector $\boldsymbol{e}_1^{\prime}$. In fact, 
let $T: \R^{\dd} \to \R^{\dd}$ be an orthogonal operator. Then, one has 
\begin{equation} \label{orthogonal} 
\varphi_{T \boldsymbol{\omega}}(x) = 
 \int_{\R^{\dd}} {\rm e}^{-\mathsf{i} \langle T \boldsymbol{\omega}, \boldsymbol{y} \rangle} \varphi (x, \|\boldsymbol{y}\|) {\rm d} \boldsymbol{y} 
= 
\int_{\R^{\dd}} {\rm e}^{-\mathsf{i} \langle \boldsymbol{\omega}, T^{-1} \boldsymbol{y} \rangle} \varphi (x, \|\boldsymbol{y}\|) {\rm d} \big ( T^{-1} \boldsymbol{y} \big ) = \varphi_{ \boldsymbol{\omega}}(x),  \end{equation} for $ \boldsymbol{\omega} \in \R^{\dd}, \; x\in[0,1)$. 
Owing to Rudin's extension theorem, there exists a mapping $\widetilde{\varphi}_{\boldsymbol{\omega}}$ that is identical to $\varphi_{\boldsymbol{\omega}}$ on $[0,1)$ and is such that $\widetilde{\varphi}_{\boldsymbol{\omega}}/\widetilde{\varphi}_{\boldsymbol{\omega}}(0)$ belongs to ${\cal P}(\R^d,\|\cdot\|)$ for every $\boldsymbol{\omega} \in \R^{\dd}$. This implies that the function $(x,y) \mapsto \widetilde{\varphi}(x,y)= {\frac{1}{(2 \pi)^{\dd}}} \int_{\R^{\dd}} {{\rm e}^{\mathsf{i} \langle \boldsymbol{\omega},y \boldsymbol{e}_1^{\prime} \rangle}} \widetilde{\varphi}_{\boldsymbol{\omega}}(x) {\rm d} \boldsymbol{\omega}$ 
 belongs to the class ${\cal P}(\R^d \times \R^{\dd},\|\cdot\|,\|\cdot\|)$. The proof is concluded by noting that $\widetilde{\varphi}$ is identical to $\varphi$ on $[0,1)\times [0,\infty)$. 


(b). The proof of this part of the statement is a straight application of Theorem 3.4 in \cite{berg-porcu} and is thus omitted. 
\end{proof}

We are now ready to provide our Rudin-type extensions for the product spaces considered in this paper. 
\begin{theorem}
\label{rudin_extension_product} 
(a). Let $\varphi: [0,1) \times [0,\infty) \to \R$ be a member of the class ${\cal P}(\B_d \times \R^{\dd},\|\cdot\|,\|\cdot\|) $ for some $d,\dd \in \mathbb{N}$, with the additional requirement that ${\varphi}(x, \|\cdot\|)$ is absolutely integrable in $\R^{\dd}$ for all $x \ge 0.$ Then, there exists a mapping $\widetilde{\varphi}:  [0,\infty)^2 \to \R$ belonging to the class ${\cal P}(\R^d \times \R^{\dd},\|\cdot\|,\|\cdot\|) $


(b). Let $\psi:[0,1) \times [0,\pi] \to \R$ be a member of the class ${\cal P} (\B_d \times \S^{\dd}, \|\cdot\|,\theta)$. Then, there exists a mapping $\widetilde{\psi}: [0,\infty) \times [0,\pi] \to \R$ 
belonging to the class ${\cal P} (\R^d \times \S^{\dd}, \|\cdot\|,\theta)$ such that $\widetilde{\psi} = \psi$ on ${[0,1)} \times [0,\pi]$. 
\end{theorem}

\begin{proof}
Let ${\varphi}$ be as asserted. We use part (a) of Theorem \ref{lemma} to claim that the mapping $\varphi_{\boldsymbol{\omega}}/\varphi_{\boldsymbol{\omega}}(0)$ as defined in (\ref{lemma_a}) belongs to the class ${\cal P}(\B_d,\|\cdot\|)$ for all $\boldsymbol{\omega} \in \R^{\dd}$. Hence, Rudin's extension theorem (see Theorem \ref{rudin}) implies that there exists a mapping $\widetilde{\varphi}_{\boldsymbol{\omega}}: [0,\infty) \to \R$ that is identical to $\varphi_{\boldsymbol{\omega}}$ on $[0,1)$ and such that $\widetilde{\varphi}_{\boldsymbol{\omega}}/\widetilde{\varphi}_{\boldsymbol{\omega}}(0)$ belongs to the class ${\cal P}(\R^d,\|\cdot\|)$ a.e. $\boldsymbol{\omega} \in \R^{\dd}$ 
Hence, we make use of (\ref{orthogonal}) in concert with classical Fourier inversion to claim that the mapping $\widetilde{\varphi}$ defined on the positive quadrant of $\R^2$ through 
$$ \widetilde{\varphi}(x, y):={\frac{1}{(2\pi)^{\dd}}} \int_{\R^{\dd}} {\rm e}^{\mathsf{i} \langle \boldsymbol{\omega}, y \boldsymbol{e}_1^{\prime} \rangle} \widetilde{\varphi}_{\boldsymbol{\omega}} {(x)} {\rm d} \boldsymbol{\omega}, \qquad x, y \ge 0, $$
with $\boldsymbol{e}_1^{\prime}$ a unit vector in $\mathbb{R}^{\dd}$
, is identical to $\varphi$ on $[0,1) \times [0,\infty)$. {This mapping belongs to the class ${\cal P}(\R^d \times \R^{\dd},\|\cdot\|,\|\cdot\|)$ owing to the Schur product theorem and the fact that the class of positive definite functions is a convex cone closed under scale mixtures.} Note that $\widetilde{\varphi}$ is well defined since $\widetilde{\varphi}_{\boldsymbol{\omega}}/\widetilde{\varphi}_{\boldsymbol{\omega}}(0) \in {\cal P}(\R^d,\|\cdot\|)$ implies $ |\widetilde{\varphi}_{\boldsymbol{\omega}}(x)| \le \widetilde{\varphi}_{\boldsymbol{\omega}}(0) {= \varphi}_{\boldsymbol{\omega}}(0)$ for all $x >0$. Hence, the existence of the integral above is directly deduced from the {existence of $\varphi(0,y)$ for any $y \geq 0$.}\\ 
(b). We provide a constructive proof. By Theorem \ref{lemma}, part (b), $\psi \in {\cal P} (\B_d \times \S^{\dd}, \|\cdot\|,\theta)$ implies the sequence {of functions} $\{ b_{n,\dd}(\cdot)\}_{n=0}^{\infty}$ to be contained in ${\cal P} (\B_d , \|\cdot\|)$. Hence, we can invoke again Theorem \ref{rudin} to claim that there exists a sequence $\{ B_{n,\dd}(\cdot)\}_{n=0}^{\infty}$ {of functions} in ${\cal P} (\R^d , \|\cdot\|)$ that is identical to $\{ b_{n,\dd}(\cdot)\}_{n=0}^{\infty}$ on $[0,1)$. Additionally, the summability of  $\{ B_{n,\dd}(\cdot)\}_{n=0}^{\infty}$ at zero is inherited from that of $\{ b_{n,\dd}(\cdot)\}_{n=0}^{\infty}$ at zero. Thus, a direct application of {classical inversion formulae and} Theorem 3.4 in \cite{berg-porcu} shows that the mapping 
$$ (x,\theta) \mapsto \widetilde{\psi}(x,\theta)=  \sum_{n=0}^{\infty} B_{n,\dd}(x) \frac{{\cal G}_{n}^{(d-1)/2} (\cos \theta)}{{\cal G}_{n}^{(d-1)/2} (1)}, \qquad x \ge 0, \theta \in [0,\pi],$$
belongs to the class ${\cal P}(\R^d \times \S^{\dd},\|\cdot\|,\theta)$ {and is identical to $\psi$ on $[0,1) \times [0,\pi)$}. 
\end{proof}
Some comments are in order. A direct inspection into Theorem 4.3.2 of \cite{sasvari1994positive} in concert with Theorem \ref{lemma} shows that radiality over the second arguments is actually not necessary, so that our extension theorem works {\em mutatis mutandis} in product spaces with radiality in $\B_d$ only. We also note that our result generalizes 4.1.9 in \cite{sasvari1994positive}, corresponding to the case of a function in ${\cal P}(\B_1 \times \R, |\cdot|,|\cdot|)$. \\

Theorem \ref{lemma} does not help finding multiradial extensions for the case $\B_d \times \B_{\dd}  $. Again, Theorem 4.3.2 in \cite{sasvari1994positive} shows that such extensions are possible, but does not allow to determine whether those can be multiradial. We certainly consider this as an open problem. \\
We finally note that Theorem \ref{rudin_extension_product} does not contradict \cite{rudin1963extension}, who proved that positive definite functions defined in hyperrectangles of $\R^d$ might not be extended to positive definite functions in $\R^d$, for $d>1$. 

In turn, Rudin's counterexample \citep{rudin1963extension} shows that positive definite functions defined in hyperrectangles might not be extendable to hyperplanes. Yet, some cases allow for a positive answer. Let $\varphi \in {\cal P}(\B_d \times \B_{\dd},\|\cdot\|,\|\cdot\|)$ such that $\varphi(x,t) = \varphi_1(x) \varphi_2(t)$, for $\varphi_1 \in {\cal P}(\B_d,\|\cdot\|)$ and $\varphi_2 \in {\cal P}(\B_{\dd},\|\cdot\|)$. Then, $\varphi_1$ extends to $\widetilde{\varphi}_1 \in {\cal P}(\R^d,\|\cdot\|)$ and $\varphi_2$ extends to $\widetilde{\varphi}_2 \in {\cal P}(\R^{\dd},\|\cdot\|)$, i.e., $\varphi_i = \widetilde{\varphi}_i$ on $[0,1)$, $i=1,2$. Hence, the product $\widetilde{\varphi}(x,t) = \widetilde{\varphi}_1(x) \widetilde{\varphi}_2(t)$ is an extension of $\varphi$ in the sense of Rudin. A similar comment applies to the function
$$ \varphi (x,t):= \int_{[0,\infty)} \int_{[0,\infty)} \Omega_d(x {u}) \Omega_{\dd}(t v) {\rm d} F(u,v), $$
for $F$ a probability measure {on the positive quadrant of $\mathbb{R}^2$}, with $\Omega_d$ as defined at (\ref{Omega}). This does not mean that the classes ${\cal P}(\B_d \times \R^{\dd}, \|\cdot\|,\|\cdot\| )$ and ${\cal P}(\R^d \times \R^{\dd}, \|\cdot\|,\|\cdot\| )$ are bijective. This is certainly true for the classes ${\cal P}(\B_d , \|\cdot\| )$ and ${\cal P}(\R^d, \|\cdot\| )$, as proved by \cite{gneiting-sasvari}.

\subsection{Uniqueness, Indeterminacy, and Measurability}
Theorem \ref{lemma} in concert with Krein's work \citep{krein1940} show that, if $\varphi_{\boldsymbol{\omega}}$ in (\ref{lemma_a}) is analytic, or if $\varphi_{\boldsymbol{\omega}}(t_o)=1$ for some $t_o \in [-1,1]$, then the extension of $\varphi \in {\cal P}(\B_1 \times \R^{\dd},|\cdot|,\|\cdot\|) $ to  ${\cal P}(\R \times \R^{\dd},|\cdot|,\|\cdot\|)$ is unique. We are not aware of any extension for higher dimensional spaces. The alternative to uniqueness is indeterminacy, which happens when there is a countable number of extensions. This would be definitely worth of a thorough investigation. \\

A beautiful result in \cite{crum} shows that, for a measurable mapping, $\varphi$, that is isotropic and positive definite in $\R^d$, with $d>1$, then $\varphi$ is continuous except possibly at zero. This proves a conjecture by \cite{schoenberg2}. The implications of such a result are illustrated by \cite{gneiting-sasvari}: from a geostatistical perspective, the restriction to measurable functions is immaterial, and practically we can write any isotropic covariance function on $\R^{d}$, $d>1$, as the sum of a pure nugget effect ({\em i.e.}, a covariance function that is identically equal to zero except for the origin) and a continuous covariance function. \cite{gneiting-sasvari} then couple Crum's result with Rudin's extension to show that every measurable isotropic correlation on $\B_d$ admits a measurable extension in the sense of Rudin. 
Measurability and Crum's decomposition might be an issue for the case of product spaces. An illustration follows: consider a {\em space-time} correlation function $\varphi \in {\cal P}(\B_d \times \B_1, \|\cdot\|, |\cdot|)$ defined as $\varphi(x,t) = \varphi_1(x) \varphi_2(t)$, for $\varphi_{1} \in {\cal P}(\B_d,\|\cdot\|)$ and $\varphi_2 \in {\cal P}(\B_1,|\cdot|)$. Clearly, the measurability theorem is not valid for $\varphi_2$ - see \cite{crum} for a counterexample. Hence, it seems that the product space case inherits the same problem emphasized by \cite{crum}.

\section{Matheron's Turning Bands Operator in Product Spaces} \label{sec5}

 Turning Bands operators are largely unexplored in product spaces. For the class ${\cal P}(\R^d \times \R^{\dd}, \|\cdot\|,\|\cdot\| )$, we invoke the arguments in \cite{porcu-mateu-christakos} to assert that $\varphi \in {\cal P}(\R^d \times \R^{\dd}, \|\cdot\|,\|\cdot\| )$ if and only if it can be written as in Equation (\ref{schoenberg2}), for a probability distribution $H$ defined on the positive quadrant of $\R^2$. Direct inspection allows to rewrite (\ref{schoenberg2}) as 
\begin{equation}
\label{TB_origin}  \varphi(x,t) = \int_{[0,\infty)} \Omega_d(xu) {\rm d} \widetilde{H}_{t,\dd}(u), \qquad x,t \ge 0,
\end{equation}
 where $\widetilde{H}_{t,\dd}(u):= \int_{[0,\infty)} \Omega_{\dd}(tv) {\rm d} H(u,v)$. One has $|\Omega_{\dd}(t)| \le \Omega_{\dd}(0)=1 $, implying $|\widetilde{H}_{t,\dd}(u)| \le \widetilde{H}_{0,\dd}(u) \le 1$, since the total mass of $H$ over $\R^2$ is identically equal to one. A similar argument implies 
$|\widetilde{H}_{t,\dd}(u)| \le \widetilde{H}_{t,\dd}(0) \le 1$. Thus, we can use Equation (\ref{TB_origin}) in concert with Fubini's theorem, as well as an integral representation for Bessel functions and Formula 9.1.20 of Abramowitz
and Stegun (1972), to get the following relation:  for $\varphi_{d,\dd} \in {\cal P}(\R^d \times \R^{\dd}, \|\cdot\|,\|\cdot\| )$, one has 
\begin{equation}
\label{TB2} \varphi_{d,\dd}(x,t) = \frac{2 \Gamma(d/2) }{\sqrt{\pi} \Gamma ((d-1)/2)} \frac{1}{x} \int_{0}^{x} \varphi_{1,\dd}(u,t) \left ( 1- \frac{u^2}{x^2}\right )^{(d-3)/2}   {\rm d} u,  
\end{equation} 
for $\varphi_{1,\dd} \in {\cal P}(\R\times \R^{\dd},|\cdot|,\|\cdot\|)$. Clearly, such an operator is in bijection between the classes $ {\cal P}(\R\times \R^{\dd},|\cdot|,\|\cdot\|)$ and ${\cal P}(\R^d \times \R^{\dd}, \|\cdot\|,\|\cdot\| )$. Hence, we can use Theorem \ref{rudin_extension_product}, part (a), while invoking the arguments in \cite{gneiting1999isotropic}, to claim that the Turning Bands operator provides a bijection between the classes $ {\cal P}(\B_1 \times \R^{\dd},|\cdot|,\|\cdot\|)$ and ${\cal P}(\B^d \times \R^{\dd}, \|\cdot\|,\|\cdot\| )$. \\

A less obvious result is to prove that such bijections hold for the class ${\cal P}(\R^d \times \S^{\dd},\|\cdot\|,\theta)$ as well. This is stated formally hereinafter. 
\begin{theorem}
\label{turning_bands_theorem} Let $d,\dd$ be positive integers. Let $\psi_{1,\dd}$ be a member of the class ${\cal P}(\R \times \S^{\dd},|\cdot|,\theta)$. Then, the Turning Bands operator in Equation (\ref{TB2}) provides a function $\psi_{d,\dd}$ being a member of the class ${\cal P}(\R^d \times \S^{\dd}, \|\cdot\|,\theta)$.
\end{theorem}
\begin{proof}
By assumption, $\psi_{1,\dd} \in {\cal P}(\R \times \S^{\dd},|\cdot|,\theta)$. This implies that there exists a sequence of functions $\{ b_{n,\dd}(\cdot)\}_{n=0}^{\infty}$ such that $b_{n,\dd}(\cdot) / b_{n,\dd}(0) \in {\cal P}(\R,|\cdot|)$ with $\sum_n b_{n,\dd}(0)=1$. Hence, we can apply Matheron's Turning Bands as in Equation (\ref{TB1}) to obtain a sequence  $\{ \widetilde{b}_{n,\dd}(\cdot)\}_{n=0}^{\infty}$ in ${\cal P}(\R^d,\|\cdot\|)$, with 
\begin{equation}
\label{jennifer} \widetilde{b}_{n,\dd}(x)= \frac{2 \Gamma(d/2) }{\sqrt{\pi} \Gamma ((d-1)/2)} \frac{1}{x} \int_{0}^{x} b_{n,\dd}(u) \left ( 1- \frac{u^2}{x^2}\right )^{(d-3)/2}   {\rm d} u.
\end{equation}  
In view of Equation (\ref{jennifer}), there exists a function $\widetilde{\psi}_{d,\dd} \in {\cal P}(\R^d \times \S^{\dd},\|\cdot\|,\theta)$ such that 
\begin{eqnarray}
 \widetilde{\psi}_{d,\dd}(x,\theta) &=& \sum_{n=0}^{\infty} \widetilde{b}_{n,\dd}(x) {\cal G}_n^{(\dd-1)/2} (\cos \theta) \nonumber \\  &=& \sum_{n=0}^{\infty} \Bigg ( \frac{2 \Gamma(d/2) }{\sqrt{\pi} \Gamma ((d-1)/2)} \frac{1}{x} \int_{0}^{x} b_{n,\dd}(u) \left ( 1- \frac{u^2}{x^2}\right )^{(d-3)/2}   {\rm d} u \Bigg )  {\cal G}_n^{(\dd-1)/2} (\cos \theta) \nonumber \\&=&
 \frac{2 \Gamma(d/2) }{\sqrt{\pi} \Gamma ((d-1)/2)} \frac{1}{x} \int_{0}^{x} \left ( 1- \frac{u^2}{x^2}\right )^{(d-3)/2} \Bigg (\sum_{n=0}^{\infty}  b_{n,\dd}(u) {\cal G}_n^{(\dd-1)/2} (\cos \theta)  \Bigg ) {\rm d} u \nonumber \\ &=&  \frac{2 \Gamma(d/2) }{\sqrt{\pi} \Gamma ((d-1)/2)} \frac{1}{x} \int_{0}^{x} \left ( 1- \frac{u^2}{x^2}\right )^{(d-3)/2} \psi_{1,\dd}(u) {\rm d} u, \nonumber 
\end{eqnarray}
where series and definite integral can be swapped because the series is absolutely convergent \citep{schoenberg1942}. The proof is completed by noting that 
\begin{eqnarray}
 \left| \sum_{n=0}^{\infty} \widetilde{b}_{n,\dd}(x) \right| &=& \left| \frac{2 \Gamma(d/2) }{\sqrt{\pi} \Gamma ((d-1)/2)} \frac{1}{x} \sum_{n=0}^{\infty}  \int_{0}^{x} b_{n,\dd}(u) \left ( 1- \frac{u^2}{x^2}\right )^{(d-3)/2}   {\rm d} u \right| \nonumber \\
 &\le & \frac{2 \Gamma(d/2) }{\sqrt{\pi} \Gamma ((d-1)/2)} \frac{1}{x} \sum_{n=0}^{\infty} b_{n,\dd}(0) \int_{0}^{x}  \left ( 1- \frac{u^2}{x^2}\right )^{(d-3)/2}   {\rm d} u \nonumber \\ &=& \frac{x \sqrt{\pi }  \Gamma \left((d-1)/2\right)}{2 \Gamma \left(d/2\right)}\frac{2 \Gamma(d/2) }{\sqrt{\pi} \Gamma ((d-1)/2)} \frac{1}{x} \sum_{n=0}^{\infty} b_{n,\dd}(0) \nonumber \\
 &=& 1, \nonumber
\end{eqnarray}
where the equality in the third line has been obtained by using formula 3.249.5 of \cite{gradshteyn2014table}. 
\end{proof}

\section{Random Fields over Balls cross Linear or Circular Time} \label{sec6}

\cite{gneiting1999isotropic} considers the function
$$
 \varphi_1(x)=\begin{cases}
    \begin{array}{cc}
      1- \alpha x & 0\le x <1 \\
      &  \\
      1-\alpha (2 -x) & 1 \le x <2, 
    \end{array}
    \end{cases}
$$
continued periodically to $[0,\infty)$ with period $2$. One has 
$$ \varphi_1(x) = 1 - \frac{\alpha}{2} + \sum_{n=1}^{\infty} \frac{4 \alpha \cos \Big ((2n-1)\pi x)}{\pi^2 (2 n -1)^2}, $$ which shows that $\varphi_1 \in {\cal P}(\R,|\cdot|)$ for $0 < \alpha \le 2$. Clearly, the restriction of $\varphi_1 $ to $[0,1)$, that we denote $\widetilde{\varphi}_1(x)=1-\alpha x$, belongs to the class ${\cal P}(\B_1,|\cdot|)$. Using the fact that the Turning Bands provides a bijection for this class, a direct computation shows that 
$$ \varphi_d (x)= 1 - \frac{\Gamma(d/2)}{\sqrt{\pi} \Gamma ((d+1)/2)} \alpha x, $$
belongs to the class  ${\cal P}(\B_d,\|\cdot\|)$. This provides an upper bound for the function $\varphi(x):= 1 - \alpha_d x$, $0\le x <1$, to belong to the class ${\cal P}(\B_d,\|\cdot\|)$.\\
The developments in previous sections allow to elaborate similar strategies for positive definite functions that are isotropic in $d$-dimensional balls and, either, symmetric over linear time ($\R$), or isotropic over circular time ($\S^1$). For instance, one can consider the function 
$$
 \varphi_{1,\dd}(x,t)=\begin{cases}
    \begin{array}{cc}
      1+\alpha(t)/2 - \alpha(t) x & 0\le x <1 \\
      &  \\
      1+\alpha(t)-\alpha(t) (2 -x) & 1 \le x <2, 
    \end{array}
    \end{cases}
$$
continued periodically to $[0,\infty)$ with period $2$, and where $\alpha \in {\cal P}( \R^{\dd},\|\cdot\|)$ or $\alpha \in {\cal P}( \S^{\dd},\theta)$. Clearly, 
$$  \varphi_{1,\dd}(x,t) = 1  + \sum_{n=1}^{\infty} \frac{4 \alpha(t) \cos \big ((2n-1)\pi x\big )}{\pi^2 (2 n -1)^2}, $$
which proves that $\varphi_{1,\dd} \in {\cal P}(\R \times \R^{\dd},|\cdot|,\|\cdot\| )$ (resp. ${\cal P}(\R \times \S^{\dd},|\cdot|,\theta )$). Hence, we can mimic the arguments in \cite{gneiting1999isotropic} in concert with Theorem \ref{turning_bands_theorem} to show that the function 
$$ \varphi_{d,\dd}(x,t)= \frac{1 + \alpha(t) \left ( \frac{1}{2} - \alpha_d x \right )}{1+ \frac{\alpha(0)}{2}}, $$
belongs to the class ${\cal P} (\B_d \times \R^{\dd}, \|\cdot\|,\|\cdot\|)$ or ${\cal P} (\B_d \times \S^{\dd}, \|\cdot\|,\theta)$ provided $\alpha \in {\cal P}( \R^{\dd},\|\cdot\|)$ or $\alpha \in {\cal P}( \S^{\dd},\theta)$, respectively. Here, $\alpha_d = \Gamma(d/2)/ \left ( \sqrt{\pi} \Gamma ((d+1)/2)\right )$.

\subsection{A worked example and plot}

As a concrete example, define $p:[0,\infty)\to \mathbb{R}$ as $p(t) = \exp(-\frac{t^2}{5})$ and  $\bar{h}:\mathbb{R}\times [0,\infty) \to \mathbb{R}$ as 

\begin{equation}
  \bar{h}(x,t)  = \frac{\exp\left({p(t) \cos (2\pi x) } \right) \cos (p(t) \sin (2 \pi x))}{{\rm e}} 
  \label{example1}
\end{equation}
Since $p^2(t) \leq 1$ for all $t$, $\bar{h}$  admits the following expansion \citep[formula 1.449.2]{gradshteyn2014table}:
\[
\bar{h}(x,t)   = \sum_{k=0}^\infty \frac{p^k(t) \cos (2\pi k x)}{k!{\rm e}}.
\]

Note that $\bar{h}(0,0)=1$, and furthermore that $p$ is positive definite on $[0,\infty)$. Given the results above, we thus wish to compute the turning bands operator, extending $\bar{h}$ to a positive definite covariance function in $\mathbb{R}^d \times [0,\infty)$: 

\[
 h_d(x,t) = \frac{2 \Gamma(d/2)}{\sqrt{\pi} \Gamma ((d-1)/2)} \frac{1}{x} \int_0^x \bar{h}(u,t) \left (1-\frac{u^2}{x^2} \right) ^{(d-3)/2}\,{\rm d}u, \quad x,t\in[0,\infty).
\]

\begin{figure}[!ht]
\includegraphics[width=15cm]{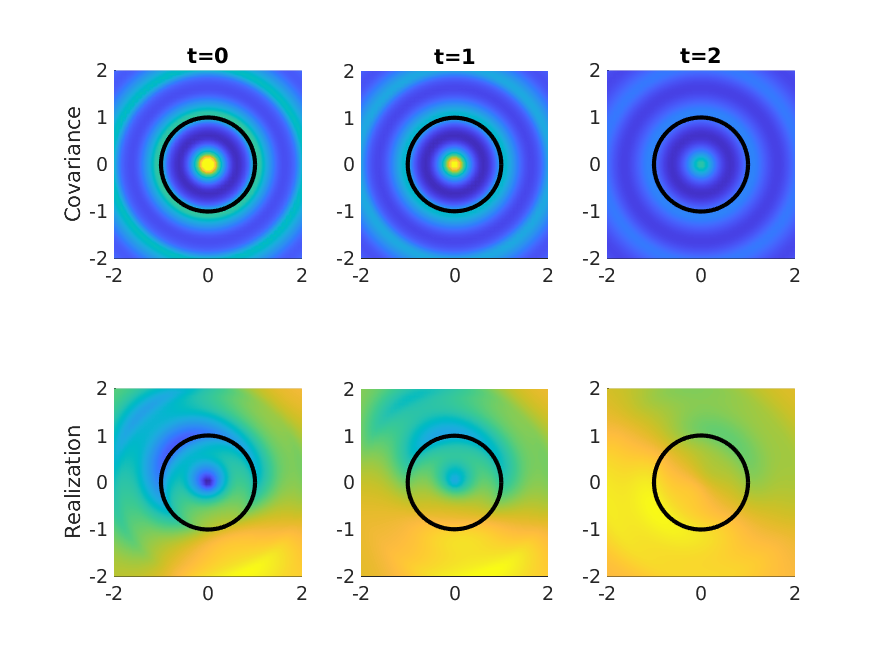}
\centering
\caption{Top: worked example extending covariance \eqref{example1} to $\mathbb{R}^2\times[0,\infty)$ via turning bands operator. Bottom: a realization of a Gaussian random field having this extended covariance at three consecutive time instants}
\label{fig:R2example}
\end{figure}

Figure \ref{fig:R2example} shows the result for $d=2$ by numerically calculating the turning bands operator and simulating the corresponding Gaussian random field. The top panels plot the entire isotropic covariance function extended to the unit disc/plane via the turning bands operator, at three different time lags ($t=0,1,2$). 
The bottom three panels are a single realization of a Gaussian Random field (in $\mathbb{R}^2 \times [0,\infty)$) plotted in the square $[-2,2]^2$ at three consecutive time instants $(0,1,2)$.
The integration was computed in MATLAB R2019b using the adaptive quadrature of \cite{shampine2008vectorized}, and the realization of the Gaussian random field was computed using the covariance matrix decomposition method \citep{davis1987production}, following an implementation of \cite{wangconstantineMATLAB}. 
The unit circle is drawn in black. The code to reproduce these results is shared at \url{https://github.com/sffeng/extension}.

\section*{Acknowledgments}
X. Emery acknowledges the funding of the National Agency for Research and Development of Chile, through grants ANID FONDECYT Regular 1210050 and ANID PIA AFB180004.  E. Porcu and S.F. Feng acknowledge this publication is based upon work supported by the Khalifa University of Science and Technology under Research Center Award No. 8474000331 (RDISC). A. Peron was partially supported by Funda\c{c}\~ao de Amparo \`a Pesquisa do Estado de S\~ao Paulo - FAPESP \# 2021/04269-0.

\bibliographystyle{apalike}
\bibliography{Rudin_EP_XE_SF_APP}

\end{document}